\documentclass{amsart}
\usepackage{amssymb,amstext,amsmath,amscd,amsthm,amsfonts,enumerate,latexsym, comment}
\usepackage{color}
\usepackage[dvipdfmx]{graphicx}
\usepackage[all]{xy}
\usepackage{tikz}

\theoremstyle{plain}
\newtheorem{thm}{Theorem}[section]
\newtheorem{theorem}[thm]{Theorem}
\newtheorem{prop}[thm]{Proposition}

\newtheorem{lem}[thm]{Lemma}

\newtheorem{cor}[thm]{Corollary}

\newtheorem*{claim*}{Claim}

\theoremstyle{definition}
\newtheorem{defn}[thm]{Definition}
\newtheorem{definition}[thm]{Definition}
\newtheorem{ex}[thm]{Example}
\newtheorem{example}[thm]{Example}
\newtheorem{rem}[thm]{Remark}

\newtheorem{fact}[thm]{Fact}

\newtheorem{setup}[thm]{Setup}

\numberwithin{equation}{section}

\newcommand{\rmr}{\mathrm{r}}

\newcommand{\fkm}{\mathfrak{m}}
\newcommand{\fkn}{\mathfrak{n}}

\newcommand{\Hom}{\operatorname{Hom}}
\newcommand{\Ext}{\operatorname{Ext}}

\newcommand{\Soc}{\operatorname{Soc}}
\newcommand{\PF}{\operatorname{PF}}
\newcommand{\pd}{\operatorname{pd}}
\newcommand{\id}{\operatorname{id}}

\newcommand{\Ann}{\operatorname{Ann}}

\newcommand{\Q}{\operatorname{Q}}

\newcommand{\conditionname}{\ensuremath{(\ast)}}

\title{On embeddings of rings into their canonical modules}

\author[T.~N.~An]{Tran Nguyen An}
\address{Department of Mathematics, Thai Nguyen University of Education,
Thai Nguyen, Vietnam}
\email{antn@tnue.edu.vn}

\author[S.~Kumashiro]{Shinya Kumashiro}
\address{Department of General education, Osaka Institute of Technology,
5-16-1 Omiya, Asahi-ku, Osaka, 535-8585, Japan}
\email{shinya.kumashiro@oit.ac.jp}

\author[M.~Samanta]{Mouma Samanta}
\address{Department of Mathematics, Indian Institute of Technology Kharagpur,
West Bengal 721302, India}
\email{mouma17@kgpian.iitkgp.ac.in}

\thanks{Kumashiro was supported by JSPS KAKENHI Grant Number JP21K13766.}
\thanks{Samanta was supported by Prime Minister's Research Fellowship,
Government of India.}

\begin{document}

\begin{abstract}
We study embeddings of Cohen--Macaulay local rings into their canonical modules such that the quotient of the cokernel by a regular sequence has the residue field as a direct summand. Motivated by almost Gorenstein rings, we prove an Ext-vanishing criterion for finite projective dimension, which implies G-regularity and the generalized Auslander--Reiten condition. We characterize this summand condition for one-dimensional fiber products and numerical semigroup rings. For Stanley--Reisner rings of graphs, we characterize the corresponding condition for graded embeddings in terms of the graph and show that it is equivalent to a strict multiplicity inequality.
\end{abstract}

\maketitle

\section{Introduction}\label{sec:introduction}

The rich duality theory of Gorenstein rings makes the search for a
theory between Cohen--Macaulay and Gorenstein rings particularly
compelling. The aim is to find classes that encompass a broad range
of examples while retaining enough of the Gorenstein structure to
support a rich structural and homological theory. Almost Gorenstein
rings offer an appealing approach to this goal: their formulation
in terms of canonical modules provides a concrete framework for
this search and has stimulated extensive research.

Let $(R,\fkm,k)$ be a Cohen--Macaulay local ring of dimension $d>0$
with canonical module $\omega_R$. Following Goto, Takahashi, and
Taniguchi \cite{GTT}, $R$ is called \emph{almost Gorenstein} if there
is an exact sequence
\begin{equation}\label{eq:intro-AG}
0\longrightarrow R\longrightarrow\omega_R\longrightarrow C
\longrightarrow0
\end{equation}
such that $e^0_{\fkm}(C)=\mu_R(C)$. Here $e^0_{\fkm}(C)$ denotes
the multiplicity of $C$ with respect to $\fkm$, and $\mu_R(C)$
denotes the number of minimal generators of $C$. In dimension one, the equality
$e^0_{\fkm}(C)=\mu_R(C)$ is equivalent to $\fkm C=0$, so the
defining sequence has a finite-dimensional $k$-vector space as
its cokernel.

Almost Gorenstein rings have been characterized or extensively studied in many classes, including numerical semigroup rings \cite{BF,GMP, Nari}, fiber products \cite{EGI}, and Stanley--Reisner rings \cite{MM}. These characterizations reveal both the scope of the theory and the restrictions imposed by almost Gorensteinness. The search for broader classes has led to several extensions, including {\it generalized Gorenstein local rings} \cite{GK} and {\it Goto rings} \cite{EndoGoto}.

To guide our search for a broader class, we turn to a result of Goto, Takahashi, and Taniguchi: every non-Gorenstein almost Gorenstein local ring is G-regular \cite[Corollary~4.5]{GTT}. We ask how far the defining condition can be weakened while preserving this property. In dimension one, the cokernel $C$ in \eqref{eq:intro-AG} is a nonzero direct sum of copies of $k$. The argument establishing G-regularity needs only one of these copies, so the same conclusion holds whenever $C$ has $k$ as a direct summand. This leads to the following definition.

\begin{definition}\label{def:CKS}
We say that $R$ satisfies condition~\conditionname{} if there exist an exact sequence
\[
0\longrightarrow R\longrightarrow\omega_R\longrightarrow C
\longrightarrow0
\]
with $C\ne0$ and a sequence $x_1,\ldots,x_{d-1}\in\fkm$ that is regular on both $R$ and $C$, such that $C/(x_1,\ldots,x_{d-1})C$ has $k$ as a direct summand.
\end{definition}

When $d=1$, the sequence is empty, so the condition means that
$C$ itself has $k$ as a direct summand.
In Theorem~\ref{thm:Ext-detection}, we prove that if $R$ satisfies
condition~\conditionname{}, then every finitely generated
$R$-module $M$ satisfying
\[
\Ext_R^i(M,R)=0\qquad(n\le i\le n+d-1)
\]
for some integer $n\ge d+2$ has finite projective dimension.
In particular, $R$ is G-regular and satisfies the generalized
Auslander--Reiten condition
(Corollary~\ref{cor:Gregular-AR}). We also prove that, when $k$ is infinite, condition~\conditionname{}
holds whenever the cokernel $C$ in an exact sequence
\eqref{eq:intro-AG} satisfies $e^0_{\fkm}(C)<2\mu_R(C)$
(Proposition~\ref{prop:local-numerical-criterion}). 

We now characterize condition~\conditionname{} for fiber products
and numerical semigroup rings, and then establish a corresponding
result for graded canonical embeddings of Stanley--Reisner rings.
Comparing these results with the known characterizations of
almost Gorenstein rings makes the extent of the enlargement explicit.

We begin with fiber products. Let $(A,\fkm,k)$ and $(B,\fkn,k)$ be Noetherian local rings. Let $f: A\to k$ and $g: B\to k$ be the canonical surjections. The subring of $A\times B$  
\[
A\times_k B := \{(a,b) \mid f(a)=g(b)\}
\]
is called {\it the fiber product of $A$ and $B$ with respect to $f$ and $g$}. 
If $A$ and $B$ are one-dimensional Cohen--Macaulay local rings, then so is $A\times_k B$ \cite[Lemma~2.1]{EGI}.

\begin{theorem}\label{thm:fiber-intro}
Let $(A,\fkm,k)$ and $(B,\fkn,k)$ be one-dimensional Cohen--Macaulay local rings, and put $R:=A\times_kB$.
Assume that $k$ is infinite, that $R$ has a canonical module,
and that $\Q(R)$ is Gorenstein. Then the following conditions
are equivalent:
\begin{enumerate}[\rm(i)]
\item $R$ satisfies condition~\conditionname{}.
\item $R$ is not Gorenstein.
\item At least one of $A$ and $B$ is not a discrete valuation ring.
\end{enumerate}
\end{theorem}

Under the same assumptions, $R=A\times_k B$ is almost Gorenstein if and only if both $A$ and $B$ are almost Gorenstein \cite[Theorem~1.1]{EGI}. 
Theorem~\ref{thm:fiber-intro} shows that condition~\conditionname{} holds for every non-Gorenstein fiber product in this setting, extending the scope well beyond almost Gorenstein fiber products.

We next consider numerical semigroup rings. A {\it numerical semigroup} is a subset $H\subseteq\mathbb N=\{0,1,2,\ldots\}$ containing $0$, closed under addition, and with finite complement in $\mathbb N$. For a field $k$, the associated {\it numerical semigroup ring} is
\[
k[[H]]:=k[[t^h \mid h\in H]]\subseteq k[[t]].
\]
This is a one-dimensional Cohen--Macaulay complete local domain with residue field $k$. The set of \emph{pseudo-Frobenius numbers} of $H$ is
\[
\PF(H):=\{a\in\mathbb N\setminus H\mid a+h\in H\text{ for all }h\in H\setminus\{0\}\}.
\]

\begin{theorem}\label{thm:semigroup-intro}
Let $H\subsetneq\mathbb N$ be a numerical semigroup and
$R=k[[H]]$. Set $\omega_R=\sum_{\gamma\in\PF(H)}Rt^{-\gamma}$, a canonical module of $R$. Then the following conditions are equivalent:
\begin{enumerate}[\rm(i)]
\item $R$ satisfies condition~\conditionname{}.
\item There exists $\delta\in\PF(H)$ such that
$\omega_R/Rt^{-\delta}$ has $k$ as a direct summand.
\item There exist $\alpha,\beta\in\PF(H)$ such that
$\alpha+\beta\in\PF(H)$.
\end{enumerate}
\end{theorem}

For comparison, we recall the characterization of almost Gorenstein numerical semigroup rings. Write $\PF(H)=\{\alpha_1<\cdots<\alpha_r\}$. Then $k[[H]]$ is almost Gorenstein if and only if $\alpha_i+\alpha_{r-i}=\alpha_r$ for all $1\le i\le r-1$ (\cite[Proposition~29]{BF}, \cite[Theorem~2.4]{Nari}).
Thus, for non-Gorenstein numerical semigroup rings, Theorem~\ref{thm:semigroup-intro} replaces this symmetry requirement by the existence of a single additive relation among pseudo-Frobenius numbers.

The final result concerns Stanley--Reisner rings of graphs. Let $\Delta$ be a finite simple graph on the vertex set $\{1,\ldots,n\}$ with at least one edge. We regard $\Delta$ as a simplicial complex whose faces are the empty set, the vertices, and the edges. Its \emph{Stanley--Reisner ring} over a field $k$ is
\[
k[\Delta]:=k[x_1,\ldots,x_n]/\left(\prod_{i\in F}x_i\ \middle|\ F\subseteq\{1,\ldots,n\},\ F\notin\Delta\right).
\]
This is a standard graded $k$-algebra of dimension two, and it is Cohen--Macaulay if and only if $\Delta$ is connected. 
A {\it bridge} of a connected graph is an edge whose deletion disconnects the graph, and a {\it cut vertex} is a vertex whose deletion disconnects it. The following theorem relates a residue-field summand condition for graded embeddings to a multiplicity inequality and the structure of the graph. Recall that for a standard graded Cohen--Macaulay $k$-algebra $R$, its \emph{$a$-invariant} is
\[
a(R)=-\min\{j\in\mathbb Z\mid[\omega_R]_j\ne0\}.
\]

\begin{theorem}\label{thm:graph-intro}
Let $k$ be an infinite field, and let $\Delta$ be a connected
finite simple graph with $n\ge3$ vertices, regarded as a
one-dimensional simplicial complex. Put $R=k[\Delta]$ and $\fkm_R=R_+$. 
Then the following conditions are equivalent:
\begin{enumerate}[\rm(i)]
\item There exist a graded exact sequence 
\[
0\longrightarrow R\longrightarrow\omega_R(-a(R))
\longrightarrow C\longrightarrow0
\]
and a linear form $f\in R_1$ regular on both $R$ and $C$ such that $C/fC$ has a graded $k$-summand.
\item There exists a graded exact sequence 
\[
0\longrightarrow R\longrightarrow\omega_R(-a(R))
\longrightarrow C\longrightarrow0
\]
such that $e^0_{\fkm_R}(C)<2\mu_R(C)$.
\item $\Delta$ is a tree with $n\ge4$, or $\Delta$ has no
bridges and has a cut vertex.
\end{enumerate}
Here $e^0_{\fkm_R}(C)$ denotes the multiplicity of $C$ with
respect to $\fkm_R$.
\end{theorem}

Matsuoka and Murai proved that $k[\Delta]$ is almost Gorenstein as a graded ring if and only if $\Delta$ is a tree or can be obtained from a cycle by repeatedly attaching another cycle along a single vertex \cite[Proposition~3.8]{MM}. Theorem~\ref{thm:graph-intro} yields an infinite family of examples beyond almost Gorenstein graded rings. 

The paper is organized as follows.
Section~\ref{sec:general} develops the general properties of embeddings of rings into their canonical modules. 
Section~\ref{sec:fiber} treats fiber products and
Theorem~\ref{thm:fiber-intro}.
Section~\ref{sec:semigroup} proves
Theorem~\ref{thm:semigroup-intro}, and
Section~\ref{sec:graph} proves Theorem~\ref{thm:graph-intro}.

\begin{setup}\label{setup15}
Throughout this paper, unless otherwise noted, $(R,\fkm,k)$ denotes a Cohen--Macaulay local ring of dimension $d>0$ with maximal ideal $\fkm$ and the residue field $k$. Suppose that $R$ possesses the canonical module $\omega_R$. We denote by $\mathrm{Q}(R)$ (resp. $\overline R$) the total quotient ring of $R$ (resp. the integral closure in $\Q(R)$). A finitely generated $R$-submodule $I$ of $\mathrm{Q}(R)$ is called {\it a fractional ideal} of $R$ if $\mathrm{Q}(R)\otimes _R I\cong \mathrm{Q}(R)$. For fractional ideals $I$ and $J$, multiplication induces an isomorphism
\[
I:J\cong \Hom_R(J, I), \qquad \alpha\longmapsto (x\longmapsto  \alpha x).
\]

Let $M$ be a finitely generated $R$-module of dimension $s$. We denote by $\ell_R(M)$ (resp. $\mu_R(M)$) the {\it length of $M$} (resp. {\it the number of minimal generators of $M$}). We write $\widehat{M}$ for the completion of $M$. 
$e^0_{\fkm}(M)$ denotes {\it the multiplicity of $M$ with respect to $\fkm$}, that is, the leading coefficient of the polynomial agreeing with $s! \cdot \ell_R(M/\fkm^{n+1}M)$ for all sufficiently large $n$.
When $M$ is Cohen--Macaulay, $\mathrm{r}_R(M):=\ell_R(\Ext_R^s(k,M))$ is called {\it the Cohen--Macaulay type of $M$}. 
\end{setup}

\section{Embeddings into canonical modules}
\label{sec:general}

\subsection{The local case}\label{subsec:local}

Let $(R,\fkm,k)$ be a Cohen--Macaulay local ring of dimension $d>0$ with canonical module $\omega_R$. 

\begin{rem}\label{rem:generically-gorenstein}
The following conditions are equivalent:
\begin{enumerate}[{\rm (i)}] 
\item There exists an injective $R$-linear homomorphism $R\to\omega_R$. 
\item $\Q(R)$ is Gorenstein. 
\item $\omega_R$ is isomorphic to an ideal of $R$ containing a nonzerodivisor.
\end{enumerate}
\end{rem}

\begin{proof}
{\rm(i)} $\Rightarrow$ {\rm(ii)}: This follows from \cite[Lemma~3.1(1),(2)]{GTT}.

{\rm(ii)} $\Rightarrow$ {\rm(iii)}: This is \cite[Proposition~3.3.18]{BH}.

{\rm(iii)} $\Rightarrow$ {\rm(i)}: Identify $\omega_R$ with an ideal $I$ containing a nonzerodivisor $a$. Then multiplication by $a$ gives an injection $R\to I\cong\omega_R$.
\end{proof}

\begin{definition}\label{def22} (Definition~\ref{def:CKS})
We say that $R$
satisfies condition~\conditionname{} if there exist an exact
sequence
\[
0\longrightarrow R\longrightarrow\omega_R\longrightarrow C
\longrightarrow0
\]
with $C\ne0$ and a sequence $x_1,\ldots,x_{d-1}\in\fkm$ that is regular on both $R$ and $C$, such that $C/(x_1,\ldots,x_{d-1})C$ has $k$ as a direct summand.
\end{definition}

\begin{rem}\label{rem:simultaneous-regular-sequence}
In Definition \ref{def22}, it suffices to assume the existence of a $C$-regular sequence $x_1,\ldots,x_{d-1}\in\fkm$ such that $C/(x_1,\ldots,x_{d-1})C$ has $k$ as a direct summand. Such a sequence can always be replaced by one that is regular on both $R$ and $C$.

Indeed, the assertion is clear when $d=1$, so assume $d\ge2$. By Remark~\ref{rem:generically-gorenstein}, we may identify $\omega_R$ with an ideal of $R$. The image of $1$ under the embedding is then a nonzerodivisor annihilating $C$. Hence $\Ann_R(C)$ contains an $R$-regular element. By Davis's lemma \cite[Exercise~16.8]{Matsumura}, we can choose $a\in\Ann_R(C)$ such that $x_1+a$ is $R$-regular. Since $aC=0$, replacing $x_1$ by $x_1+a$ preserves its action on $C$. Repeating this argument on the successive quotients proves the assertion.
\end{rem}

\begin{rem}\label{lem:gorenstein-condition-regular}
$R$ is Gorenstein and satisfies
condition~\conditionname{} if and only if $R$ is regular.
\end{rem}

\begin{proof}
If $R$ is regular, choose a regular system of parameters
$x_1,\ldots,x_d$.
The embedding $R\xrightarrow{x_1}R\cong\omega_R$ has
cokernel $R/x_1R$, whose quotient by $x_2,\ldots,x_d$ is $k$.
Hence $R$ satisfies condition~\conditionname{}.

Conversely, suppose that $R$ is Gorenstein and satisfies
condition~\conditionname{}.
Since $\omega_R\cong R$, there are a nonzerodivisor
$u\in\fkm$ and an $R/uR$-regular sequence
$x_2,\ldots,x_d\in\fkm$ such that
$R/(u,x_2,\ldots,x_d)$ has $k$ as a direct summand.
Since $R/(u,x_2,\ldots,x_d)$ is cyclic, so it is isomorphic to $k$. 
Thus $\fkm=(u,x_2,\ldots,x_d)$, and $R$ is regular.
\end{proof}

\begin{theorem}\label{thm:Ext-detection}
Suppose that $R$ satisfies condition~\conditionname{}.
Let $M$ be a finitely generated $R$-module.
If there is an integer $n\ge d+2$ such that
\[
\Ext_R^i(M,R)=0\qquad(n\le i\le n+d-1),
\]
then $\pd_R M<\infty$.
\end{theorem}

\begin{proof}
Choose an exact sequence
\[
0\longrightarrow R\longrightarrow\omega_R\longrightarrow C\longrightarrow0
\]
and a $C$-regular sequence $x_1,\ldots,x_{d-1}\in\fkm$ such that $C/(x_1,\ldots,x_{d-1})C$ has $k$ as a direct summand.

Applying $\Hom_R(M,-)$ to the above exact sequence gives
\[
\Ext_R^j(M,\omega_R)\longrightarrow\Ext_R^j(M,C)\longrightarrow\Ext_R^{j+1}(M,R)\longrightarrow\Ext_R^{j+1}(M,\omega_R).
\]
Since $\id_R\omega_R=d$, by the hypothesis and the inequality $n-1>d$,
\[
\Ext_R^j(M,C)\cong\Ext_R^{j+1}(M,R)=0\qquad(n-1\le j\le n+d-2).
\]

By \cite[Lemma~2.3(3)]{DG}, it follows that
\[
\Ext_R^{n-1}\bigl(M,C/(x_1,\ldots,x_{d-1})C\bigr)=0.
\]
Since the quotient has $k$ as a direct summand, $\Ext_R^{n-1}(M,k)=0$, and hence $\pd_R M<\infty$.
\end{proof}

We recall the generalized Auslander--Reiten condition and G-regularity.

\begin{defn}[{\cite[Lemma~2.1]{DKS}}]\label{def:GARC}
We say that $R$ satisfies the \emph{generalized Auslander--Reiten condition} (GARC) if the following equivalent conditions hold.
\begin{enumerate}[\rm(i)]
\item Every finitely generated $R$-module $M$ satisfying $\Ext_R^i(M,M\oplus R)=0$ for all $i\gg0$ has finite projective dimension.
\item The \emph{symmetric Auslander condition} (SAC) holds: for every finitely generated $R$-module $M$, 
\[
\Ext_R^i(M,R)=0\quad(i>0),\qquad \Ext_R^i(M,M)=0\quad(i\gg0)
\]
implies $\Ext_R^i(M,M)=0$ for all $i>0$.
\end{enumerate}
\end{defn}

\begin{defn}[{\cite{TakahashiGregular}}]\label{def:Gregular}
The ring $R$ is called \emph{G-regular} if every finitely generated totally reflexive $R$-module is free. Here a finitely generated $R$-module $M$ is \emph{totally reflexive} if $M$ is reflexive and $\Ext_R^i(M,R)=\Ext_R^i(\Hom_R(M,R),R)=0$ for all $i>0$.
\end{defn}

As a direct consequence of Theorem~\ref{thm:Ext-detection}, we get the following. 

\begin{cor}\label{cor:Gregular-AR}
A Cohen--Macaulay local ring satisfying condition~\conditionname{} is G-regular and satisfies the
generalized Auslander--Reiten condition.
\end{cor}

The following lemma gives a criterion for an $R$-module to have $k$ as a direct summand.

\begin{lem}\label{lem:local-summand}
Let $M$ be a finitely generated $R$-module. The following conditions are equivalent:
\begin{enumerate}[\rm(i)]
\item $M$ has $k$ as a direct summand;
\item there exists $y\in M\setminus\fkm M$ such that $\fkm y=0$;
\item $\Soc(M)\not\subseteq\fkm M$.
\end{enumerate}
\end{lem}

\begin{proof}
The implications {\rm(i)}$\Rightarrow${\rm(ii)} and {\rm(ii)}$\Leftrightarrow${\rm(iii)} are immediate. To prove {\rm(ii)}$\Rightarrow${\rm(i)}, let $\iota:k\to M$ be the injection sending $1$ to $y$. Since the image of $y$ in $M/\fkm M$ is nonzero, there is a $k$-linear functional $M/\fkm M\to k$ sending it to $1$. Its composite with $M\to M/\fkm M$ is an $R$-linear retraction of $\iota$, so $k$ is a direct summand of $M$.
\end{proof}

Throughout the rest of this subsection, suppose that there is an exact sequence
\begin{equation}\label{eq:local-canonical-embedding}
0\longrightarrow R\xrightarrow{\varphi}\omega_R
\longrightarrow C\longrightarrow0
\end{equation}
with $C\ne0$.

\begin{rem}[{\cite[Lemma~3.1(2)]{GTT}}]\label{rem:local-cokernel-dimension}
The cokernel $C$ in \eqref{eq:local-canonical-embedding} is a Cohen--Macaulay $R$-module of dimension $d-1$.
\end{rem}

\begin{lem}\label{prop:local-self-duality}
For the cokernel $C$ in \eqref{eq:local-canonical-embedding}, there is an isomorphism 
\[
\Ext_R^1(C,\omega_R)\cong C.
\]
In particular, $\mathrm{r}_R(C)=\mu_R(C)$.
\end{lem}

\begin{proof}
Applying $\Hom_R(-,\omega_R)$ to
\eqref{eq:local-canonical-embedding} gives
\[
0\longrightarrow R\xrightarrow{\varphi}\omega_R
\longrightarrow\Ext_R^1(C,\omega_R)\longrightarrow0
\]
by Remark \ref{rem:local-cokernel-dimension} and \cite[Proposition 3.3.3]{BH}. 
The equality $\mathrm{r}_R(C)=\mu_R(C)$ follows by \cite[Proposition 3.3.11]{BH}. 
\end{proof}

\begin{lem}\label{prop:mu-type-minus-one}
Assume that $d=1$ and that $R$ is not Gorenstein.
If $C$ in \eqref{eq:local-canonical-embedding}
has $k$ as a direct summand, then $\varphi(1)\notin\fkm\omega_R$. Equivalently, 
\[
\mu_R(C)=\mathrm{r}(R)-1.
\]
\end{lem}

\begin{proof}
By Remark~\ref{rem:generically-gorenstein}, we may identify $\omega_R$ with an ideal of $R$. Since $\varphi(1)$ is a nonzerodivisor, replacing $\omega_R$ by $\varphi(1)^{-1}\omega_R$, we may assume that $R\subseteq\omega_R\subseteq\Q(R)$ and that $\varphi$ is the inclusion. By Lemma~\ref{lem:local-summand}, there exists $q\in\omega_R$ such that
\[
q\fkm\subseteq R,\qquad q\notin\fkm\omega_R+R.
\]
If $q\fkm\not\subseteq\fkm$, then $q\fkm=R$. Thus $\fkm=q^{-1}R$ is principal, forcing $R$ to be regular, contrary to the hypothesis. Hence $q\fkm\subseteq\fkm$. If $1\in\fkm\omega_R$, then $q\in q\fkm\omega_R\subseteq\fkm\omega_R$, a contradiction. Therefore $\varphi(1)=1\notin\fkm\omega_R$.

Since $1\notin\fkm\omega_R$ and $C=\omega_R/R$, we have $\mu_R(C)=\mu_R(\omega_R)-1=\mathrm{r}(R)-1$.
\end{proof}

\begin{rem}
The non-Gorenstein hypothesis in
Lemma~\ref{prop:mu-type-minus-one} is essential.
For $R=k[[t]]$, the exact sequence
\[
0\longrightarrow R\xrightarrow{\cdot t}\omega_R=R
\longrightarrow k\longrightarrow0
\]
has cokernel $k$, but $\mu_R(k)=1\ne0=\mathrm{r}(R)-1$.
\end{rem}

\begin{prop}\label{prop:local-numerical-criterion}
Assume that $k$ is infinite.
If the cokernel $C$ in \eqref{eq:local-canonical-embedding} satisfies
\[
e^0_{\fkm}(C)<2\mu_R(C),
\]
then there exist $x_1,\ldots,x_{d-1}\in\fkm$ forming a regular
sequence on both $R$ and $C$ such that $C/(x_1,\ldots,x_{d-1})C$ has $k$ as a direct summand.
\end{prop}

\begin{proof}
We prove the assertion by induction on $d$. Suppose first that $d=1$. Since $C$ has finite length and $\mathrm{r}_R(C)=\mu_R(C)$, we have
\[
\ell_R(\fkm C)=\ell_R(C)-\ell_R(C/\fkm C)=e^0_{\fkm}(C)-\mu_R(C)<\mu_R(C)=\mathrm{r}_R(C)=\ell_R(\Soc(C)).
\]
Thus $\Soc(C)\not\subseteq\fkm C$, and Lemma~\ref{lem:local-summand} shows that $C$ has $k$ as a direct summand.

Suppose $d>1$. Since $k$ is infinite, choose $x\in\fkm$ superficial for both $R$ and $C$. Both modules are Cohen--Macaulay of positive dimension, so $x$ is regular on both. Thus, the exact sequence \eqref{eq:local-canonical-embedding} induces an exact sequence $0 \to R/xR \to \omega_{R/xR} \to C/xC \to 0$, and
\[
e^0_{\fkm}(C/xC)=e^0_{\fkm}(C)<2\mu_R(C)=2\mu_R(C/xC).
\]
The induction hypothesis over $R/xR$ therefore gives a sequence regular on both $R/xR$ and $C/xC$ whose quotient of $C/xC$ has $k$ as a direct summand. Taking $x_1=x$ and lifting this sequence to $x_2,\ldots,x_{d-1}\in\fkm$ proves the assertion.
\end{proof}

\begin{cor}\label{cor:AG-implies-condition}
Assume that $k$ is infinite. If $R$ is non-Gorenstein and almost Gorenstein, then $R$ satisfies condition~\conditionname{}.
\end{cor}

\begin{proof}
Since $R$ is not Gorenstein, a defining exact sequence \eqref{eq:intro-AG}  has a nonzero cokernel $C$ satisfying $e^0_{\fkm}(C)=\mu_R(C)<2\mu_R(C)$. 
The assertion follows from Proposition~\ref{prop:local-numerical-criterion}.
\end{proof}

For a fixed canonical embedding, the converse of Proposition~\ref{prop:local-numerical-criterion} need not hold. 

\begin{example}\label{ex:canonical-nonconverse}
Let $H=\langle5,7,8,11\rangle$ and $R=k[[H]]$ (see Section~\ref{sec:semigroup} for the notation). Then $\PF(H)=\{3,6,9\}$, and $\omega_R=Rt^{-3}+Rt^{-6}+Rt^{-9}$ is a canonical module of $R$. Set $C=\omega_R/Rt^{-6}$. The class of $t^{-3}$ in $C$ is a minimal generator annihilated by $\fkm$, since $3\in\PF(H)$. Hence $C$ has $k$ as a direct summand by Lemma~\ref{lem:local-summand}. On the other hand, one can check that 
\[
e^0_{\fkm}(C)=\ell_R(C)=5>4=2\mu_R(C).
\]
However, for a different embedding, we have
\[
\omega_R/Rt^{-9}\cong k^{\oplus2}.
\]
Thus $R$ itself is almost Gorenstein, and the multiplicity inequality holds for this latter cokernel.
\end{example}

\subsection{The graded case}\label{subsec:graded}

In this subsection, we work in the graded setting.

\begin{setup}
Let $R=\bigoplus_{i\ge0}R_i$ be a standard graded Cohen--Macaulay algebra of dimension $d>0$ over an infinite field $k=R_0$, and set $\fkm=R_+$. Let $M=\bigoplus_{i\in \mathbb{Z}} M_i$ be a finitely generated graded $R$-module of dimension $s$. Then, 
\[
\operatorname{Hilb}_M(t):=\sum_{j\in\mathbb Z}\dim_k M_j\cdot t^j
\]
is called the {\it Hilbert series of $M$}. 
The multiplicity $e^0_{\fkm}(M)$ of $M$ with respect to $\fkm$ is the leading coefficient of the polynomial agreeing with $(s-1)!\cdot\dim_k M_n $ for all sufficiently large $n$ when $s>0$. If $s=0$, we set $e^0_{\fkm}(M)=\ell_R(M)$. This agrees with the Hilbert--Samuel multiplicity $e^0_{\fkm R_{\fkm}}(M_{\fkm})$ in the local sense. We denote by $M^\vee$ the graded $k$-dual of $M$, whose homogeneous components are
\[
(M^\vee)_j=\Hom_k(M_{-j},k)\qquad(j\in\mathbb Z).
\]
Let $\omega_R$ denote the graded canonical module of $R$. Recall that the \emph{$a$-invariant} of $R$ is
\[
a(R):=-\min\{j\in\mathbb Z\mid[\omega_R]_j\ne0\}.
\]
\end{setup}

\begin{rem}\label{rem:graded-cokernel}
Suppose that there is a graded exact sequence
\[
0\longrightarrow R\xrightarrow{\varphi}\omega_R(-a(R))
\longrightarrow C\longrightarrow0
\]
with $C\ne0$. Then, the following hold. 

\begin{enumerate}[{\rm (1)}] 
\item $C$ is Cohen--Macaulay of dimension $d-1$.
\item $\varphi(1)$ is part of minimal generators of $\omega_R(-a(R))$. Hence, $\mu_R(C)=\mathrm{r}(R)-1$. 
\end{enumerate}
\end{rem}

The following lemma will be used in Section~\ref{sec:graph}.

\begin{lem}\label{lem:graded-two-degrees}
Let $M$ be a finite-dimensional graded $R$-module
concentrated in degrees zero and one.
Suppose that $M^\vee\cong M(1)$.
Then $M$ has a graded $k$-summand if and only if
\[
\ell_R(M)<2\mu_R(M).
\]
\end{lem}

\begin{proof}
The isomorphism $M^\vee\cong M(1)$ gives $\dim_k M_0=\dim_k M_1$. Since $\fkm M\subseteq M_1$, we have
\begin{align*}
2\mu_R(M)-\ell_R(M)
=&2\left(\dim_k M_0+\dim_k(M_1/\fkm M)\right)-2\dim_k M_0 \\
=&2\dim_k(M_1/\fkm M).
\end{align*}
Thus the inequality $\ell_R(M)<2\mu_R(M)$ is equivalent to $\fkm M\subsetneq M_1$.

If $\fkm M\subsetneq M_1$, choose $y\in M_1\setminus\fkm M$. Since $\fkm M_1=0$, Lemma~\ref{lem:local-summand}, applied in the graded setting, shows that $Ry\cong k(-1)$ is a direct summand of $M$.

Conversely, any graded $k$-summand of $M$ is isomorphic to either $k$ or $k(-1)$. If $k$ is a direct summand, dualizing and using $M\cong M^\vee(-1)$ shows that $k(-1)$ is also a direct summand. Hence $M$ has a direct summand $k(-1)$, whose generator lies in $M_1\setminus\fkm M$. Thus $\fkm M\subsetneq M_1$.
\end{proof}

\begin{prop}\label{prop:graded-numerical-criterion}
Suppose that there is a graded exact sequence
\[
0\longrightarrow R\longrightarrow\omega_R(-a(R))
\longrightarrow C\longrightarrow0
\]
with $C\ne0$. Consider the following conditions. 
\begin{enumerate}[{\rm (i)}] 
\item $e^0_{\fkm}(C)<2\mu_R(C)$.
\item There exists a regular sequence $f_1,\ldots,f_{d-1}\in R_1$ on both $R$ and $C$ such that $C/(f_1,\ldots,f_{d-1})C$ has a graded $k$-summand. 
\end{enumerate}
Then, {\rm (i)} $\Rightarrow$ {\rm (ii)} holds. {\rm (ii)} $\Rightarrow$ {\rm (i)} also holds if  $a(R)\le2-d$. 

\end{prop}

\begin{proof}
For any sequence $f_1,\ldots,f_{d-1}\in R_1$ regular on both $R$ and $C$, set $\overline R=R/(f_1,\ldots,f_{d-1})R$ and $M=C/(f_1,\ldots,f_{d-1})C$. Since $C$ is Cohen--Macaulay of dimension $d-1$, we have
\[
\ell_R(M)=e^0_{\fkm}(C),
\qquad
\mu_R(M)=\mu_R(C).
\]
Since $\omega_{\overline R}\cong(\omega_R/(f_1,\ldots,f_{d-1})\omega_R)(d-1)$, reducing the canonical embedding modulo the regular sequence gives
\[
0\longrightarrow\overline R\longrightarrow\omega_{\overline R}(-a(R)-d+1)\longrightarrow M\longrightarrow0.
\]
Applying graded $\Hom_{\overline R}(-,\omega_{\overline R})$ to this sequence, as in Lemma~\ref{prop:local-self-duality}, and using graded local duality \cite[Theorem~3.6.19]{BH}, we obtain
\begin{equation}\label{eq2191}
M^\vee\cong\Ext_{\overline R}^1(M,\omega_{\overline R})\cong M(a(R)+d-1).
\end{equation}
In particular, $\ell_R(\Soc(M))=\mu_R(M)$.

{\rm(i)} $\Rightarrow$ {\rm(ii)}: Since $k$ is infinite, we can choose a sequence of linear forms regular on both $R$ and $C$. By the above argument, condition {\rm(i)} gives
\[
\ell_R(\fkm M)=\ell_R(M)-\mu_R(M)
<\mu_R(M)=\ell_R(\Soc(M)).
\]
Thus $\Soc(M)\not\subseteq\fkm M$. Choosing a homogeneous element in $\Soc(M)\setminus\fkm M$ and applying the graded version of Lemma~\ref{lem:local-summand} proves {\rm(ii)}.

{\rm(ii)} $\Rightarrow$ {\rm(i)}: Assume that $a(R)\le2-d$, and choose the sequence as in {\rm(ii)}. For every $j$ with $M_j\ne0$, we have $j\ge0$, and \eqref{eq2191} gives
\[
0\ne\Hom_k(M_j,k)=(M^\vee)_{-j}\cong M_{a(R)+d-1-j}.
\]
Thus
\[
0\le j\le a(R)+d-1\le1.
\]
Since $M\ne0$, it follows that $a(R)+d-1$ is either $0$ or $1$. 
If $a(R)=1-d$, then $M$ is concentrated in degree zero, so
\[
e^0_{\fkm}(C)=\ell_R(M)=\mu_R(M)=\mu_R(C)<2\mu_R(C).
\]
If $a(R)=2-d$, then $M^\vee\cong M(1)$ and $M$ is concentrated in degrees zero and one. Since $M$ has a graded $k$-summand, Lemma~\ref{lem:graded-two-degrees} gives
\[
e^0_{\fkm}(C)=\ell_R(M)<2\mu_R(M)=2\mu_R(C).
\]
This proves {\rm(i)}.
\end{proof}

Recall that $R$ is called an \emph{almost Gorenstein graded ring} if there exists a graded exact sequence
\[
0\longrightarrow R\longrightarrow\omega_R(-a(R))\longrightarrow C\longrightarrow0
\]
such that $e^0_{\fkm}(C)=\mu_R(C)$ (\cite{GTT}). 

\begin{cor}\label{cor:graded-AG-implies-summand}
Suppose that $R$ is a non-Gorenstein almost Gorenstein graded ring, and choose a graded exact sequence
\[
0\longrightarrow R\longrightarrow\omega_R(-a(R))\longrightarrow C\longrightarrow0
\]
such that $e^0_{\fkm}(C)=\mu_R(C)$. Then, there exists a sequence $f_1,\ldots,f_{d-1}\in R_1$ regular on both $R$ and $C$ such that $C/(f_1,\ldots,f_{d-1})C$ has a graded $k$-summand.
\end{cor}

\begin{proof}
Since $R$ is not Gorenstein, we have $C\ne0$, and hence $e^0_{\fkm}(C)=\mu_R(C)<2\mu_R(C)$.
The assertion follows from Proposition~\ref{prop:graded-numerical-criterion}.
\end{proof}

\section{Fiber products}\label{sec:fiber}

Let $(A,\fkm,k)$ and $(B,\fkn,k)$ be Noetherian local rings
with a common residue field $k$, and let $f:A\to k$ and
$g:B\to k$ be the canonical surjections. The subring
\[
R=A\times_k B
=\{(a,b)\in A\times B\mid f(a)=g(b)\}
\]
is called the \emph{fiber product of $A$ and $B$ over $k$}.
It is known that $R$ is a Noetherian local ring with maximal ideal
$\fkm_R:=\fkm\times\fkn$ and residue field $k$
\cite[Lemma~2.1(1),(2)]{EGI}.
By definition, there is an exact sequence of $R$-modules
\begin{equation}\label{eq:fiber-basic-sequence}
0\longrightarrow R\longrightarrow A\times B
\xrightarrow{\delta}k\longrightarrow0,
\qquad \text{where } \delta(a,b)=f(a)-g(b).
\end{equation}

Throughout the rest of this section, assume that $A$ and
$B$ are one-dimensional Cohen--Macaulay local rings.
Then $R$ is also a one-dimensional Cohen--Macaulay local
ring \cite[Lemma~2.1(3)]{EGI}.
The exact sequence above gives the following.

\begin{lem}\label{prop:fiber-canonical-sequence}
The following conditions are equivalent:
\begin{enumerate}[\rm(i)]
\item $R$ has a canonical module;
\item $A$ and $B$ have canonical modules.
\end{enumerate}
When these conditions hold, there is an exact sequence
of $R$-modules
\[
0\longrightarrow\omega_A\oplus\omega_B
\longrightarrow\omega_R\longrightarrow k\longrightarrow0.
\]
\end{lem}

\begin{proof}
Assume {\rm(i)}. Since we have the natural surjective projections
\[
\pi_A:R\longrightarrow A,\quad (a,b)\longmapsto a,
\qquad
\pi_B:R\longrightarrow B,\quad (a,b)\longmapsto b,
\]
$\Hom_R(A,\omega_R)$ and $\Hom_R(B,\omega_R)$ are canonical modules of $A$ and $B$, respectively.
This proves {\rm(ii)}.

Applying $\Hom_R(-,\omega_R)$ to
\eqref{eq:fiber-basic-sequence}, we obtain
\[
0\longrightarrow\omega_A\oplus\omega_B
\longrightarrow\omega_R
\longrightarrow k
\longrightarrow0.
\]
This is the asserted exact sequence.

Conversely, assume {\rm(ii)}. Completing \eqref{eq:fiber-basic-sequence} and using
the exactness of completion, we obtain $\widehat R\cong\widehat A\times_k\widehat B$. 
Since $\widehat R$ has a canonical module, applying the preceding argument to $\widehat R$ therefore gives
\begin{align}\label{eq32}
0\longrightarrow \omega_{\widehat A}\oplus\omega_{\widehat B}
\longrightarrow\omega_{\widehat R}
\longrightarrow k\longrightarrow0.
\end{align}

On the other hand, since $\Ext_R^1(k,\omega_A\oplus\omega_B)$ is annihilated by $\fkm_R$ and $R/\fkm_R \cong \widehat{R}/\fkm_R \widehat{R}$, we get 
\begin{align*}
\Ext_R^1(k,\omega_A\oplus\omega_B)&\cong (\widehat{R}/\fkm_R \widehat{R})\otimes_R \Ext_R^1(k,\omega_A\oplus\omega_B) \\
&\cong\widehat R\otimes_R\Ext_R^1(k,\omega_A\oplus\omega_B)\\
&\cong \Ext_{\widehat R}^1(k,\omega_{\widehat{A}}\oplus\omega_{\widehat{B}}).
\end{align*}
Under this isomorphism, let $\xi\in\Ext_R^1(k,\omega_A\oplus\omega_B)$ correspond
to the extension class of \eqref{eq32}. 
Choose an exact sequence
\[
0\longrightarrow \omega_A\oplus\omega_B\longrightarrow M\longrightarrow k
\longrightarrow0
\]
representing $\xi$, where $M$ is an $R$-module. By construction, $\widehat M\cong\omega_{\widehat R}$.
Hence $M$ is a canonical module of $R$.
\end{proof}

\begin{rem}\label{rem:fiber-colon-unit}
The following hold true.
\begin{enumerate}[\rm(i)] 
\item (\cite[Setting~4.1]{EGI}) $\Q(R)\cong\Q(A)\times\Q(B)$. In particular, $\Q(R)$ is Gorenstein if and only if both $\Q(A)$ and $\Q(B)$ are Gorenstein.
\item Suppose that $k$ is infinite, $R$ has a canonical module, and $\Q(R)$ is Gorenstein. 
\begin{enumerate}[{\rm (a)}] 
\item (\cite[Setting~4.1]{EGI}) We can choose fractional canonical ideals $K$ of $A$ and $L$ of $B$ such that
\[
A\subseteq K\subseteq\overline A,
\qquad
B\subseteq L\subseteq\overline B.
\]
\item We further suppose that $A$ is not a DVR. Then, there exists
$h_0\in(A:\fkm)\setminus K$ such that
$K:\fkm=K+Ah_0$ (\cite[Section~4.1, before Lemma~4.2]{EGI}).
We may choose such a generator $h$ with $h\fkm=\fkm$.
Indeed, $A:\fkm=\fkm:\fkm$ is a finite semilocal $A$-algebra. Fix a maximal ideal $\mathfrak M$ of $\fkm:\fkm$. If $h_0+c$ and $h_0+c'$ both belong to $\mathfrak M$ for $c,c'\in A$, then
\[
c-c'\in\mathfrak M\cap A=\fkm.
\]
Thus, to ensure that $h_0+c\notin\mathfrak M$, we need to avoid at most one residue class of $c$ modulo $\fkm$.

Since $\fkm:\fkm$ has only finitely many maximal ideals and $k$ is infinite, we can choose $c\in A$ such that $h:=h_0+c$ belongs to none of these maximal ideals. Hence $h$ is a unit of $\fkm:\fkm$. That is, $h\in \fkm:\fkm$ and $h^{-1}\in \fkm:\fkm$, so $h\fkm=\fkm$. Moreover, since $c\in A\subseteq K$, we still have
\[
h\in(A:\fkm)\setminus K,\qquad K:\fkm=K+Ah.
\]
The same assertion holds for $B$ and $L$ when $B$ is not a DVR.
\end{enumerate}
\end{enumerate}
\end{rem}



We now prove Theorem~\ref{thm:fiber-intro}.

\begin{proof}[Proof of Theorem~\ref{thm:fiber-intro}]
{\rm(ii)} $\iff$ {\rm(iii)}: This is \cite[Proposition~2.2(3)]{EGI}. 

{\rm(i)} $\Rightarrow$ {\rm(ii)}: Note that $R$ is not regular since $\mu_R(\fkm_R)=\mu_A(\fkm)+\mu_B(\fkn)\ge2$ by \cite[Proposition~2.2(1)]{EGI}. By Remark~\ref{lem:gorenstein-condition-regular}, we have the assertion.

{\rm(iii)} $\Rightarrow$ {\rm(i)}: We may assume that $B$ is not a DVR. Choose $K$ and $L$ as in Remark~\ref{rem:fiber-colon-unit}(ii)(a). By Remark~\ref{rem:fiber-colon-unit}(ii)(b), choose $g_2\in(B:\fkn)\setminus L$ such that
\[
L:\fkn=L+Bg_2,\qquad g_2\fkn=\fkn.
\]

\smallskip
\noindent
{\bf Case 1 ($A$ is not a DVR).} 
Choose $g_1\in(A:\fkm)\setminus K$ similarly, with $K:\fkm=K+Ag_1$ and $g_1\fkm=\fkm$. Set
\[
\psi=(g_1,g_2),\qquad X=(K\times L)+R\psi.
\]
By \cite[Lemma~4.2]{EGI}, we have $X\cong\omega_R$. Since $\fkm_R\psi=\fkm\times\fkn\subseteq K\times L$ and $\psi\notin K\times L$, we obtain
\[
X/(K\times L)\cong k,\qquad
(K\times L)\cap R\psi=\fkm_R\psi=\fkm\times\fkn.
\]
Both $g_1$ and $g_2$ are invertible in their respective total quotient rings, so multiplication by $\psi$ defines an embedding $R\to X$. Its cokernel is
\[
C:=X/R\psi\cong(K\times L)/(\fkm\times\fkn)
\cong K/\fkm\oplus L/\fkn.
\]
Since $K\subseteq\overline A$ and $\overline A$ is integral over $A$, we have $\fkm K\subseteq\fkm\overline A\ne\overline A$, so $1\notin\fkm K$. Thus the class of $1$ in $K/\fkm$ is annihilated by $\fkm$ and lies outside $\fkm(K/\fkm)$. By Lemma~\ref{lem:local-summand}, $K/\fkm$, and therefore $C$, has $k$ as a direct summand.

\smallskip
\noindent
{\bf Case 2 ($A$ is a DVR).}
Then $K=A$. Let $\fkm=xA$, and set
\[
\psi=(x^{-1},g_2),\qquad X=(A\times L)+R\psi.
\]
Since $\fkm_R\psi=A\times\fkn\subseteq A\times L$ and $\psi\notin A\times L$, we have
\[
X/(A\times L)\cong k,\qquad
(A\times L)\cap R\psi=A\times\fkn.
\]

We first compute $X:\fkm_R$ to show that $X\cong\omega_R$. For $(u,v)\in\Q(A)\times\Q(B)$, we have
\[
\begin{aligned}
(u,v)\in X:\fkm_R
&\Longleftrightarrow
(u\fkm)\times0\subseteq X
\ \text{and}\ 0\times(v\fkn)\subseteq X\\
&\Longleftrightarrow
u\fkm\subseteq A
\ \text{and}\ v\fkn\subseteq L.
\end{aligned}
\]
For the last equivalence, write an element of $X$ as
\[
(s,t)+(a,b)\psi=(s+ax^{-1},\,t+bg_2),
\qquad (s,t)\in A\times L,\quad (a,b)\in R.
\]
If either coordinate is zero, then $a\in xA = \fkm$ or $b\in\fkn$, respectively, because $x^{-1}\notin A$ and $g_2\notin L$. Since $(a,b)\in R$, either condition implies both, and hence the element belongs to $A\times L$. The converse follows from $A\times L\subseteq X$. Therefore,
\[
X:\fkm_R=(A:\fkm)\times(L:\fkn).
\]
Since $(A:\fkm)/A\cong k$, $(L:\fkn)/L\cong k$, and $X/(A\times L)\cong k$, we have
\[
\ell_R((X:\fkm_R)/X)=\ell_R((X:\fkm_R)/(A\times L))-\ell_R(X/(A\times L))=2-1=1.
\]
Hence $(X:\fkm_R)/X\cong k$.

On the other hand, applying $\Hom_R(-,X)$ to $0\to\fkm_R\to R\to k\to0$, we get $(X:\fkm_R)/X\cong\Ext_R^1(k,X)$. Therefore $\rmr_R(X)=1$. Since $R\subseteq X\subseteq\Q(R)$, $X$ is faithful and maximal Cohen--Macaulay, so \cite[Proposition~3.3.13]{BH} gives $X\cong\omega_R$.

Multiplication by the invertible element $\psi\in\Q(R)$ now gives an embedding $R\to X$ with cokernel
\[
C:=X/R\psi\cong(A\times L)/(A\times\fkn)\cong L/\fkn.
\]
The same argument as in Case~1, applied to $B\subseteq L\subseteq\overline B$, shows that $L/\fkn$ has $k$ as a direct summand. 

Therefore, $R$ satisfies condition~\conditionname{} in both cases.
\end{proof}

The assumption that $k$ is infinite in Theorem~\ref{thm:fiber-intro} cannot be omitted, as the following example shows.

\begin{example}\label{ex:fiber-finite-field}
Let $k=\mathbb Z/2\mathbb{Z}$ and set
\[
R:=\bigl(k[[x,y]]/(xy)\bigr)\times_k k[[z]]
\cong k[[x,y,z]]/(xy,xz,yz).
\]
Then $R$ is a non-Gorenstein one-dimensional Cohen--Macaulay local ring with a canonical module, and $\Q(R)$ is Gorenstein. However, $R$ does not satisfy condition~\conditionname{}.
\end{example}

\begin{proof}
The ring $R$ is complete, reduced, and one-dimensional Cohen--Macaulay, so it has a canonical module and $\Q(R)$ is Gorenstein. Moreover,
\[
R/(x+y+z)R\cong k[[x,y]]/(x,y)^2,
\]
which gives $\mathrm{r}(R)=2$.

Suppose that $R$ satisfies condition~\conditionname{}, and choose an exact sequence $0\to R\to\omega_R\to C\to0$ such that $C$ has $k$ as a direct summand. By Lemma~\ref{prop:mu-type-minus-one}, we have $\mu_R(C)=\mathrm{r}(R)-1=1$. Hence $C\cong k$, so $R$ is almost Gorenstein. This is a contradiction since $R$ is not almost Gorenstein (see \cite[Remark~3.5]{GTT}).
\end{proof}

\section{Numerical semigroup rings}\label{sec:semigroup}

Recall that a {\it numerical semigroup} is a subset $H\subseteq\mathbb N=\{0,1,2,\ldots\}$ containing $0$, closed under addition, and with finite complement in $\mathbb N$. 
Throughout this section, let $H\subsetneq\mathbb N$ be a numerical semigroup and let $k$ be a field. The subring
\[
R:=k[[H]]=k[[t^h\mid h\in H]]
\]
of the formal power series ring $k[[t]]$ is called the {\it numerical semigroup ring} of $H$ over $k$. We write
\[
\PF(H):=\{a\in\mathbb N\setminus H\mid a+h\in H\text{ for all }h\in H\setminus\{0\}\}
\]
for the set of {\it pseudo-Frobenius numbers} of $H$.
    
\begin{fact}(\cite[Example~(2.1.9)]{Goto-Watanabe})
\[
\omega_R:=\sum_{\gamma\in\PF(H)}Rt^{-\gamma}
\] 
is a canonical module of $R$.
\end{fact}

\begin{lem}\label{lem:2}
$(R:\mathfrak{m})/R \cong \sum_{\alpha \in \PF(H)}kt^{\alpha}$.
\end{lem}

\begin{proof}
Since $R$ is not a DVR, we have $R:\fkm=\fkm:\fkm\subseteq\overline R=k[[t]]$. As $\fkm$ is generated by monomials, so is $R:\fkm$. For $a\in \mathbb{N}$, $t^a \in R:\mathfrak{m}$ if and only if $a\in H\cup \PF(H)$ by the definition of $\PF(H)$. Therefore,
\[
R:\fkm=R\oplus\bigoplus_{\alpha\in\PF(H)}kt^\alpha
\]
as $k$-vector spaces, proving the assertion.
\end{proof}

\begin{proof}[Proof of Theorem~\ref{thm:semigroup-intro}]
{\rm(iii)} $\Rightarrow$ {\rm(ii)}: Choose $\alpha,\beta,\delta\in\PF(H)$ with $\alpha+\beta=\delta$. Multiplication by $t^\delta$ gives
\[
\omega_R/Rt^{-\delta}\cong t^\delta\omega_R/R.
\]
The class of $t^\alpha=t^{\delta-\beta}$ is part of minimal generators of this quotient and is annihilated by $\fkm$ (Lemma~\ref{lem:2}). Hence Lemma~\ref{lem:local-summand} gives the assertion.

    {\rm (ii)} $\Rightarrow$ {\rm (i)}: Assuming that (ii) holds, 
    \[
    0 \to R \xrightarrow{\cdot t^{-\delta}} \omega_R \to \omega_R/Rt^{-\delta} \to 0
    \]
    gives the desired short exact sequence.

{\rm(i)} $\Rightarrow$ {\rm(iii)}: Write $C\cong\omega_R/Rg$ with $0\ne g\in\omega_R$. By Lemma~\ref{lem:local-summand}, choose $f\in\omega_R\setminus(\fkm\omega_R+Rg)$ such that $\fkm f\subseteq Rg$. Then $f/g\in R:\fkm$, so Lemma~\ref{lem:2} gives
\[
f/g=r+\sum_{\alpha\in\PF(H)}d_\alpha t^\alpha
\qquad(r\in R,\ d_\alpha\in k).
\]
Since $rg\in Rg$, we still have $f-rg\notin\fkm\omega_R+Rg$ and $\fkm(f-rg)\subseteq Rg$. Replacing $f$ by $f-rg$, we may therefore assume that $r=0$. Write
\[
g=\sum_{\delta\in\PF(H)}b_\delta t^{-\delta},
\qquad
f=\sum_{\gamma\in\PF(H)}c_\gamma t^{-\gamma},
\qquad \text{where \quad} b_\delta,c_\gamma\in R.
\]
Since $f\notin\fkm\omega_R$, some $c_\beta$ is a unit. For $\gamma\in\PF(H)$ with $\gamma\ne\beta$, we have $\gamma-\beta\notin H$: otherwise, $\gamma-\beta\in H\setminus\{0\}$ and therefore $\gamma=\beta+(\gamma-\beta)\in H$, a contradiction. Since $c_\gamma\in R$, it follows that $c_\gamma$ has no term of degree $\gamma-\beta$, so $c_\gamma t^{-\gamma}$ contributes nothing to the coefficient of $t^{-\beta}$. Thus the coefficient of $t^{-\beta}$ in $f$ is precisely the constant term of $c_\beta$, which is nonzero because $c_\beta$ is a unit.

Comparing this coefficient in
\[
\sum_{\gamma\in\PF(H)}c_\gamma t^{-\gamma}
=
\left(\sum_{\delta\in\PF(H)}b_\delta t^{-\delta}\right)
\left(\sum_{\alpha\in\PF(H)}d_\alpha t^\alpha\right),
\]
we obtain $-\beta=h-\delta+\alpha$ for some $h\in H$ and $\delta,\alpha\in\PF(H)$. If $h>0$, then $\beta+h\in H\setminus\{0\}$, and hence $\delta=\alpha+(\beta+h)\in H$, a contradiction. Therefore $h=0$, so $\delta=\beta+\alpha\in\PF(H)$.
\end{proof}

The canonical embedding satisfying condition~\conditionname{} need not be unique. 

\begin{ex}
Let $H=\langle 7,8,10,11,13\rangle$ and $R=k[[H]]$. Then
\[
\PF(H)=\{3,6,9,12\},
\qquad
\omega_R=Rt^{-3}+Rt^{-6}+Rt^{-9}+Rt^{-12}.
\]
Both $\omega_R/Rt^{-9}$ and $\omega_R/Rg$ have $k$ as a direct summand, where $g=t^{-12}+ct^{-9}$ for each $c\in k$.
\end{ex}

\begin{proof}
It is routine to check that $\PF(H)=\{3,6,9,12\}$ and $\omega_R=Rt^{-3}+Rt^{-6}+Rt^{-9}+Rt^{-12}$. Since $3\in\PF(H)$, the class of $t^{-6}$ in $\omega_R/Rt^{-9}$ is part of minimal generators annihilated by $\fkm$. By Lemma~\ref{lem:local-summand}, $\omega_R/Rt^{-9}$ has $k$ as a direct summand. 

For $\omega_R/Rg$, set $K:=t^{12}\omega_R=(1,t^3,t^6,t^9)$. Then $\omega_R/Rg\cong K/R(1+ct^3)$. 
Take $f=t^3(1+ct^3)=t^3+ct^6\in K$. Since $\fkm t^3\subseteq R$, we have $\fkm f\subseteq R(1+ct^3)$. Moreover, since $1,t^3,t^6,t^9$ minimally generate $K$, their classes form a $k$-basis of $K/\fkm K$. With respect to this basis, the classes of $1+ct^3$ and $f$ have coordinate vectors $(1,c,0,0)$ and $(0,1,c,0)$, respectively, which are linearly independent.
Hence
\[
f\notin\fkm K+R(1+ct^3),
\]
and Lemma~\ref{lem:local-summand} shows that the class of $f$ generates a $k$-summand.
\end{proof}

Theorem~\ref{thm:semigroup-intro} provides a wealth of examples satisfying condition~\conditionname{}, including infinite families of rings that are not almost Gorenstein, as illustrated below.

\begin{example}\label{ex:class-H-main-thm}
    Let $a\ge 1$ be an integer. Set $H_a:=\{0, 2a+1, 4a+1 \rightarrow\}\cup A$, where $A$ is a subset of $\{2a+2,\dots,4a-1\}$. Then we have the following. 
    \begin{enumerate}[\rm(i)]
        \item $H_a$ is a numerical semigroup.
        \item $2a,4a\in\PF(H_a)$. Hence $k[[H_a]]$ satisfies condition~\conditionname{}.
\item If $a\ge2$ and $4a-1\notin A$, then $k[[H_a]]$
is not almost Gorenstein.
    \end{enumerate}
\end{example}

\begin{proof}
{\rm(i)}
The sum of any two positive elements of $H_a$ is at least $4a+2$, and hence belongs to $H_a$. Thus $H_a$ is closed under addition. Since $0\in H_a$ and $\mathbb N\setminus H_a$ is finite, $H_a$ is a numerical semigroup.

{\rm(ii)}
Since $2a,4a\notin H_a$ and $2a+h,4a+h\ge4a+1$ for every $h\in H_a\setminus\{0\}$, we have $2a,4a\in\PF(H_a)$.
The last assertion follows from $2a+2a=4a$.

{\rm(iii)}
Suppose that $a\ge2$ and $4a-1\notin A$.
Then $4a-1\in\PF(H_a)$, and the Frobenius number of $H_a$ is $4a$.
If $H_a$ were almost symmetric, then
\cite[Theorem~2.4]{Nari} would give $1\in\PF(H_a)$.
Starting from $2a+1\in H_a$ and repeatedly adding $1$,
we would obtain $4a-1\in H_a$, a contradiction.
Hence $k[[H_a]]$ is not almost Gorenstein.
\end{proof}

\begin{example}
    Let $b\ge 2$ be an integer. Set $H_b:=\{0, 2b, 4b-2, 4b \rightarrow\}\cup B$, where $B$ is a subset of $\{2b+2,\dots,4b-3\}$. Then we have the following. 
    \begin{enumerate}[\rm(i)]
        \item $H_b$ is a numerical semigroup.
        \item $2b-2,2b+1,4b-1\in\PF(H_b)$. Hence $k[[H_b]]$ satisfies condition~\conditionname{}.
        \item If $b\ge3$ and $4b-4,4b-3\notin B$, then $k[[H_b]]$
is not almost Gorenstein.
    \end{enumerate}
\end{example}

\begin{proof}
{\rm(i)}
The sum of any two positive elements of $H_b$ is at least $4b$, and hence belongs to $H_b$. Thus $H_b$ is closed under addition. Since $0\in H_b$ and $\mathbb N\setminus H_b$ is finite, $H_b$ is a numerical semigroup.

{\rm(ii)}
The integers $2b-2,2b+1,4b-1$ do not belong to $H_b$.
Their sums with positive elements of $H_b$ are at least
$4b$, except for $(2b-2)+2b=4b-2\in H_b$.
Thus $2b-2,2b+1,4b-1\in\PF(H_b)$.
The last assertion follows from $(2b-2)+(2b+1)=4b-1$.

{\rm(iii)}
Suppose that $b\ge3$ and $4b-4,4b-3\notin B$.
Then $4b-3\in\PF(H_b)$, and the Frobenius number of $H_b$ is $4b-1$.
If $H_b$ were almost symmetric, then
\cite[Theorem~2.4]{Nari} would give $2\in\PF(H_b)$.
Starting from $2b\in H_b$ and repeatedly adding $2$,
we would obtain $4b-4\in H_b$, a contradiction.
Hence $k[[H_b]]$ is not almost Gorenstein.
\end{proof}

\section{Stanley--Reisner rings of graphs}\label{sec:graph}

A \emph{finite simple graph} consists of a finite set $V$ of vertices and a set of two-element subsets of $V$, called edges. Let $k$ be a field, and let $\Delta$ be such a graph with at least one edge. We regard $\Delta$ as a one-dimensional simplicial complex whose faces are the empty set, the vertices, and the edges. The ring
\[
k[\Delta]:=k[x_v\mid v\in V]\Big/\left(\prod_{v\in F}x_v\ \middle|\ F\subseteq V,\ F\notin\Delta\right)
\]
is called the \emph{Stanley--Reisner ring} of $\Delta$ over $k$.

Throughout this section, assume that $k$ is infinite and $\Delta$ is connected. Set $R=k[\Delta]$ and $\fkm=R_+$, and let $n=|V|\ge2$ and $m$ denote the numbers of vertices and edges of $\Delta$, respectively. By Reisner's criterion, $R$ is a Cohen--Macaulay standard graded $k$-algebra of dimension two.

\begin{lem}\label{lem:graph-Hilbert-series}
The following hold.
\begin{enumerate}[{\rm(i)}]
\item $\displaystyle \operatorname{Hilb}_R(t)=\frac{1+(n-2)t+(m-n+1)t^2}{(1-t)^2}$.
\item $\displaystyle \operatorname{Hilb}_{\omega_R}(t)=\frac{(m-n+1)+(n-2)t+t^2}{(1-t)^2}$.
\end{enumerate}
\end{lem}

\begin{proof}
{\rm(i)} follows from \cite[Theorem 5.1.7]{BH}. {\rm(ii)} follows from (i) and \cite[Corollary 4.4.6(a)]{BH}.
\end{proof}

We introduce the terminology and notation used in the statement and proof of Theorem~\ref{thm:graph-intro}.

\begin{definition}\label{def:graph-terminology}
A \emph{bridge} is an edge whose deletion disconnects $\Delta$, and a \emph{cut vertex} is a vertex whose deletion disconnects $\Delta$. For each vertex $v\in V$, we write $\Delta-v$ for the graph obtained by deleting $v$ and all edges incident to $v$, and denote its number of connected components by $c(\Delta-v)$.
\end{definition}

\begin{rem}{\rm ({\cite[Corollary~2.3(1) and equation~(3)]{DS}})}\label{prop:graph-type}
Assume that $n\ge3$. Then
\[
\mathrm{r}(R)=m-n+1+\sum_{v\in V}\bigl(c(\Delta-v)-1\bigr).
\]
\end{rem}

%

\begin{rem}\label{rem:graph-embeddings}
The following hold.
\begin{enumerate}[{\rm(i)}]
\item (\cite[Proposition~3.8]{MM}): $R$ is an almost Gorenstein graded ring if and only if $\Delta$ is a tree or can be obtained from a cycle by repeatedly attaching another cycle along a single vertex.
\item There exists a degree-zero embedding $R\to\omega_R$ if and only if $\Delta$ has no bridges.
\end{enumerate}
\end{rem}

\begin{proof}
{\rm(ii)} This follows from \cite[Theorem~1.1(i)$\iff$(ii)]{MM} and Reisner's criterion.
\end{proof}

We now prove Theorem~\ref{thm:graph-intro}.

\begin{proof}[Proof of Theorem~\ref{thm:graph-intro}]
{\rm(i)} $\Longleftrightarrow$ {\rm(ii)}: By Lemma~\ref{lem:graph-Hilbert-series}(ii), we have $a(R)\le0$. Since $\dim R=2$, Proposition~\ref{prop:graded-numerical-criterion} gives the equivalence when $C\ne0$. If $C=0$, neither condition holds.

{\rm(ii)} $\Longleftrightarrow$ {\rm(iii)}: 
Note that $m-n+1\ge0$ since $\Delta$ is connected. 

Suppose first that $m-n+1=0$. Then $\Delta$ is a tree, so $R$ is almost Gorenstein as a graded ring by \cite[Proposition~3.8]{MM}. For each vertex $v$, the number $c(\Delta-v)$ equals the number of edges incident to $v$. Thus $\sum_{v\in V}c(\Delta-v)=2m$, and Remark~\ref{prop:graph-type} gives
\[
\mathrm{r}(R)=\sum_{v\in V}\bigl(c(\Delta-v)-1\bigr)=2m-n=n-2.
\]
If $n=3$, then $R$ is Gorenstein, so every degree-zero embedding $R\to\omega_R(-a(R))$ is an isomorphism and {\rm(ii)} fails. If $n\ge4$, then $R$ is not Gorenstein, and Corollary~\ref{cor:graded-AG-implies-summand} gives {\rm(i)}, hence {\rm(ii)}. Thus {\rm(ii)} holds exactly when $n\ge4$.

Now suppose that $m-n+1>0$. By Lemma~\ref{lem:graph-Hilbert-series}(ii), we have $a(R)=0$. Remark~\ref{rem:graph-embeddings}(ii) shows that a degree-zero embedding $R\to\omega_R$ exists if and only if $\Delta$ has no bridges. If $\Delta$ has a bridge, neither {\rm(ii)} nor {\rm(iii)} holds.

Assume therefore that $\Delta$ has no bridges, and choose a graded exact sequence $0\to R\to\omega_R\to C\to0$. Lemma~\ref{lem:graph-Hilbert-series} gives
\begin{equation}\label{eq:graph-Hilbert-C}
\begin{aligned}
\operatorname{Hilb}_C(t)
&=\operatorname{Hilb}_{\omega_R}(t)-\operatorname{Hilb}_R(t)\\
&=\frac{(m-n)(1-t^2)}{(1-t)^2}
=\frac{(m-n)(1+t)}{1-t}.
\end{aligned}
\end{equation}
Together with Remark~\ref{rem:graded-cokernel}(2), this yields
\begin{equation}\label{eq:graph-numerical-data}
e^0_{\fkm}(C)=2(m-n),
\qquad
\mu_R(C)=\mathrm{r}(R)-1.
\end{equation}
These equalities also hold when $C=0$, since then $m=n$ and $\mathrm{r}(R)=1$. Consequently,
\begin{align*}
& e^0_{\fkm}(C)<2\mu_R(C)
\quad \Longleftrightarrow \quad \mathrm{r}(R)>m-n+1 \\
 \Longleftrightarrow \quad & \sum_{v\in V}\bigl(c(\Delta-v)-1\bigr)>0 \quad \Longleftrightarrow \quad \Delta\text{ has a cut vertex},
\end{align*}
where the second equivalence follows from Remark~\ref{prop:graph-type}.
\end{proof}

Theorem~\ref{thm:graph-intro} applies to a broad class of graphs beyond those giving almost Gorenstein graded rings. 

\begin{example}\label{ex:graph-beyond-AG}
Let $\Gamma$ be a connected graph on at least three vertices with no cut vertices that is not a cycle. Examples include the graphs in Figure~\ref{fig:graph-examples} and the complete graphs $K_q$ for all $q\ge4$. Starting from $\Gamma$, repeatedly attach a connected graph on at least three vertices with no cut vertices by identifying exactly one of its vertices with a vertex of the graph already constructed. After at least one attachment, the resulting graph $\Delta$ has no bridges and has a cut vertex, so it satisfies condition~{\rm(iii)} of Theorem~\ref{thm:graph-intro}. Since the initial graph $\Gamma$ has no cut vertices and is not a cycle, $\Delta$ cannot be obtained by repeatedly attaching cycles along single vertices. Thus $k[\Delta]$ is not almost Gorenstein as a graded ring (see Remark~\ref{rem:graph-embeddings}(i)).

\begin{figure}[htbp]
\centering
\begin{tikzpicture}[scale=.9,transform shape,line width=.6pt,every node/.style={font=\small}]

\begin{scope}[shift={(3.6,0)}]
\draw (-1.1,0)--(0,1)--(1.1,0);
\draw (-1.1,0)--(0,0)--(1.1,0);
\draw (-1.1,0)--(0,-1)--(1.1,0);
\foreach \p in {(-1.1,0),(1.1,0),(0,1),(0,0),(0,-1)}
  \fill \p circle (1.5pt);
\node at (0,-1.4) {};
\end{scope}

\begin{scope}
\draw (-1.1,-.85)--(1.1,-.85)--(0,1.05)--cycle;
\draw (0,-.22)--(-1.1,-.85);
\draw (0,-.22)--(1.1,-.85);
\draw (0,-.22)--(0,1.05);
\foreach \p in {(-1.1,-.85),(1.1,-.85),(0,1.05),(0,-.22)}
  \fill \p circle (1.5pt);
\end{scope}

\begin{scope}[shift={(7.2,0)}]
\coordinate (A) at (-1,-.85);
\coordinate (B) at (.45,-.85);
\coordinate (C) at (.45,.6);
\coordinate (D) at (-1,.6);
\coordinate (E) at (-.35,-.3);
\coordinate (F) at (1.1,-.3);
\coordinate (G) at (1.1,1.15);
\coordinate (H) at (-.35,1.15);

\draw[densely dashed] (A)--(E)--(F);
\draw[densely dashed] (E)--(H);

\draw (A)--(B)--(C)--(D)--cycle;
\draw (B)--(F)--(G)--(H)--(D);
\draw (C)--(G);

\foreach \p in {A,B,C,D,E,F,G,H}
  \fill (\p) circle (1.5pt);

\node at (0,-1.4) {};
\end{scope}

\end{tikzpicture}
\caption{}
\label{fig:graph-examples}
\end{figure}
\end{example}

%
%



\end{document}